\documentclass[a4paper, 12pt]{article}
\usepackage{amsmath}
\usepackage{amssymb}
\usepackage{amsthm}
\usepackage[top=20truemm, left=20truemm, right=20truemm]{geometry}

\usepackage{amscd}
\newtheorem{dfn}{Definition}[section]
\newtheorem{thm}[dfn]{Theorem}
\newtheorem{lem}[dfn]{Lemma}
\newtheorem{rem}[dfn]{Remark}
\newtheorem{cor}[dfn]{Corollary}

\newtheorem{prop}[dfn]{Proposition}

\usepackage{amscd}
\usepackage{graphicx}
\usepackage[all]{xy}

\title{A topological invariant for continuous fields of Cuntz algebras II}
\author{Taro Sogabe \\\small Graduate School of Science, Kyoto University, Japan\\
\small staro@math.kyoto-u.ac.jp}

\begin{document}
\maketitle
\abstract
We investigate an invariant for continuous fields of the Cuntz algebra $\mathcal{O}_{n+1}$ introduced in \cite{S}, and find a way to obtain a continuous field of $\mathbb{M}_n(\mathcal{O}_\infty)$ from that of $\mathcal{O}_{n+1}$ using the construction of the invariant.
By Brown's representability theorem,
this gives a bijection from the set of the isomorphism classes of continuous fields of $\mathcal{O}_{n+1}$ to those of $\mathbb{M}_n(\mathcal{O}_\infty)$.
As a consequence,
we obtain a new proof for M. Dadarlat's classification result of continuous fields of $\mathcal{O}_{n+1}$ arising from vector bundles,
which corresponds to those of $\mathbb{M}_n(\mathcal{O}_\infty)$ stably isomorphic to the trivial field.

\section{Introduction}
Our purpose is to investigate the invariant for continuous fields of the Cuntz algebras introduced in \cite{S}.
The Cuntz algebra $\mathcal{O}_{n+1}$ is a typical example of a Kirchberg algebra,
and its continuous fields over a finite CW-complex $X$ are classified in \cite{D1} when the cohomology groups $H^*(X, \mathbb{Z})$ do not admit $n-$torsion.
All continuous fields classified in \cite{D1} are constructed via the Cuntz--Pimsner algebras,
and it is proved in \cite{D1} that, for general $X$, not every continuous field of $\mathcal{O}_{n+1}$ is given by the Cuntz--Pimsner construction.

Recently,
M. Dadarlat and U. Pennig introduced a generalized cohomology $E_D^*$ in \cite{DP} for every strongly self-absorbing C*-algebra $D$ satisfying the UCT including the infinite Cuntz algebra $\mathcal{O}_\infty$.
In \cite{S},
using the reduced cohomology $\bar{E}^*_{\mathcal{O}_\infty}$,
we define an invariant $\mathfrak{b}_{\mathcal{O}_\infty}$ of continuous fields of $\mathcal{O}_{n+1}$,
and show that a continuous field $\mathcal{O}$ over a finite CW-complex $X$ is given via the Cuntz--Pimsner algebra if and only if $\mathfrak{b}_{\mathcal{O}_\infty}([\mathcal{O}])=0\in \bar{E}^1_{\mathcal{O}_\infty}(X)$.
By \cite{D2},
the set of the isomorphism classes of continuous fields of $\mathcal{O}_{n+1}$ over a finite CW-complex $X$ is identified with the homotopy set $[X, \operatorname{BAut}(\mathcal{O}_{n+1})]$,
and the invariant is a map $\mathfrak{b}_{\mathcal{O}_\infty} : [X, \operatorname{BAut}(\mathcal{O}_{n+1})]\to \bar{E}^1_{\mathcal{O}_\infty}(X)$.

In this paper,
we construct a natural transformation $$T_X : [X, \operatorname{BAut}(\mathcal{O}_{n+1})]\to [X, \operatorname{BAut}(\mathbb{M}_n(\mathcal{O}_\infty))]$$ which turns out to be bijective by Brown's representability theorem (Theorem \ref{M}).
Since the homotopy set $[X, \operatorname{BAut}(\mathbb{M}_n(\mathcal{O}_\infty))]$ is identified with the set of the isomorphism classes of locally trivial continuous fields of $\mathbb{M}_n(\mathcal{O}_\infty)$,
the map $\operatorname{B}(\eta)_* : [X, \operatorname{BAut}(\mathbb{M}_n(\mathcal{O}_\infty))]\ni[\mathcal{B}]\mapsto [\mathbb{K}\otimes\mathcal{B}]\in \bar{E}^1_{\mathcal{O}_\infty}(X)$ is defined (see Section 2.2),
and we have $-\mathfrak{b}_{\mathcal{O}_\infty}=\operatorname{B}(\eta)_*\circ T_X$.

Since the inverse image $\mathfrak{b}_{\mathcal{O}_\infty}^{-1}(0)$ is equal to the set of the Cuntz--Pimsner algebras of vector bundles classified in \cite{D1} (see \cite[{Sec. 4}]{S}),
M. Dadarlat's classification result \cite[{Th. 5.3}]{D1}, which enables us to count the cardinality of the set of the isomorphism classes of those Cuntz--Pimsner algebras, can be proved by counting the cardinality of the set of the isomorphism classes of the continuous fields of $\mathbb{M}_n(\mathcal{O}_\infty)$ stably isomorphic to the trivial field (see Remark \ref{dd}, Corollary \ref{ds}).


We also investigate the map $T_{SX}$ for the reduced suspension $SX$,
and show that the map gives a group isomorphism $T_{SX} : [X, \operatorname{Aut}(\mathcal{O}_{n+1})]\to [X, \operatorname{Aut}(\mathbb{M}_n(\mathcal{O}_\infty))]$ (Corollary \ref{si}).
\section*{Acknowledgements}
The author would like to thank Prof. M. Izumi for suggesting to investigate the $K^0(X)$-module homomorphism in Remark \ref{m} and for his support and encouragement.
The author also would like to thank Prof. U. Pennig for helpful discussions about Lemma \ref{g}, \ref{key}, Remark \ref{m} and for many stimulating conversations.

\section{Preliminaries}
\subsection{Notation}
Let $\mathbb{K}$ be the C*-algebra of compact operators on the separable infinite dimensional Hilbert space,
and let $\mathbb{M}_n$ be the n by n matrix algebra.
For a C*-algebra $A$,
we denote by $K_i(A)$ the i-th K-group and denote by $[p]_0\in K_0(A)$ (resp. $[u]_1\in K_1(A)$) the class of the projection $p$ (resp. the unitary $u$).
If $A$ is unital,
we denote by $1_A$ the unit and by $U(A)$ the group of unitary elements. 
Let $C(X)$ be the C*-algebra of all continuous functions on $X$.
We write $K^i(X)=K_i(C(X)), \tilde{K}^i(X)=K_i(C_0(X, x_0))$ where $C_0(X, x_0)$ is the set of functions vanishing at $x_0\in X$.
We refer to \cite{B} for the K-groups.

Let $(X, x_0)$ and $(Y, y_0)$ be two pointed finite CW-complexes,
and let $[X, Y]$ (resp. $[X, Y]_0$) be the set of the homotopy classes of the continuous maps (resp. the set of the base point preserving homotopy classes of the base point preserving continuous maps).
The i-th homotopy group of $X$ is $\pi_i(X):=[S^i, X]_0$,
and, for $Y$ with $\pi_0(Y)=\pi_1(Y)=0$,
the natural map $[X, Y]_0\to [X, Y]$ is bijective (see \cite[{Th. 6.57}]{AT}).

We refer to \cite{N, D2} for the definition of the continuous $C(X)$-algebras.
Let $\mathcal{P}\to X$ be a principal $\operatorname{Aut}(A)$ bundle,
and let $\mathcal{A}$ be the associated bundle $\mathcal{P}\times_{\operatorname{Aut}(A)}A$.
Then,
the section algebra $\Gamma(X, \mathcal{A})$ is a locally trivial continuous $C(X)$-algebra.
We identify the section algebra with the associated bundle,
and write $\mathcal{A}$ by abuse of notation.
Let $\mathcal{A}_x$ denote the image of the evaluation map $ev_x : \mathcal{A}\to\mathcal{A}/C_0(X, x)\mathcal{A}\cong A$.
We always assume that the fiber $\mathcal{A}_x$ is nuclear.
Since those principal $\operatorname{Aut}(A)$ bundles are classified by $[X, \operatorname{BAut}(A)]$,
we denote the $C(X)$-linear isomorphism class of the $C(X)$-algebra by $[\mathcal{A}]\in [X, \operatorname{BAut}(A)]$.
For two locally trivial continuous $C(X)$-algebras $\mathcal{A}$, $\mathcal{B}$,
we can define the tensor product $\mathcal{A}\otimes_{C(X)}\mathcal{B}$ (see \cite{Bl}).

Let $E_{n+1}$ be the universal C*-algebra, called the Cuntz--Toeplitz algebra, generated by $n+1$ isometries with mutually orthogonal ranges.
It is known that the unital map $\mathbb{C}\to E_{n+1}$ is a KK-equivalence.
Let $\{T_i\}_{i=1}^{n+1}$ be the canonical generators, and let $e:=1-\sum_{i=1}^{n+1}T_iT_i^*$ be the minimal projection which generates the only non-trivial ideal of $E_{n+1}$ isomorphic to $\mathbb{K}$.
The quotient algebra $\mathcal{O}_{n+1}:=E_{n+1}/\mathbb{K}$ is called the Cuntz algebra.
We denote by $\mathcal{O}_\infty$ the universal C*-algebra generated by countably infinite isometries with mutually orthogonal ranges.
The inclusion $\mathbb{K}\to E_{n+1}$ gives the map 
$K_0(\mathbb{K})=\mathbb{Z}\xrightarrow{-n} K_0(E_{n+1})=\mathbb{Z}$,
and one has
$$K_0(\mathcal{O}_{n+1})=\mathbb{Z}_n,\quad K_0(\mathcal{O}_\infty)=\mathbb{Z},\quad K_1(\mathcal{O}_{n+1})=K_1(\mathcal{O}_\infty)=0$$
(see \cite{C}).
One has $\mathcal{O}_{n+1}\cong\mathcal{O}_{n+1}\otimes\mathcal{O}_\infty$ and $\mathcal{O}_{n+1}\cong (E_{n+1}\otimes\mathcal{O}_\infty)/(\mathbb{K}\otimes\mathcal{O}_\infty)$ by the classification result of the Kirchberg algebras.

The homotopy groups of $\operatorname{Aut}(\mathcal{O}_{n+1})$ and $\operatorname{Aut}(\mathbb{M}_n(\mathcal{O}_\infty))$ are given by \cite[{Th. 5.9}]{D4} :
$$\pi_{2k}(\operatorname{Aut}(\mathcal{O}_{n+1}))=\pi_{2k}(\operatorname{Aut}(\mathbb{M}_n(\mathcal{O}_\infty)))=0, \;\; \pi_{2k+1}(\operatorname{Aut}(\mathcal{O}_{n+1}))=\pi_{2k+1}(\operatorname{Aut}(\mathbb{M}_n(\mathcal{O}_\infty)))=\mathbb{Z}_n,$$
$k\geq 0$.
Thus,
two sets $[S^k, \operatorname{BAut}(\mathcal{O}_{n+1})], [S^k, \operatorname{BAut}(\mathbb{M}_n(\mathcal{O}_\infty))]$ have the same cardinality, and one has $[X, \operatorname{BAut}(\mathcal{O}_{n+1})]_0=[X, \operatorname{BAut}(\mathcal{O}_{n+1})]$, $[X, \operatorname{BAut}(\mathbb{M}_n(\mathcal{O}_{\infty}))]_0=[X, \operatorname{BAut}(\mathbb{M}_n(\mathcal{O}_{\infty}))]$.


\subsection{The Dadarlat--Pennig theory}
We briefly explain the cohomology group $\bar{E}_{\mathcal{O}_\infty}^1(X)$.
Recall that $\mathcal{O}_\infty$ is a strongly self-absorbing C*-algebra,
in other words,
there is a continuous path of unitary $\{u_t\}_{t\in [0, 1)}\subset \mathcal{O}_\infty^{\otimes 2}$ and an isomorphism $\phi : \mathcal{O}_\infty\to\mathcal{O}_\infty^{\otimes 2}$ satisfying $\lim_{t\to 1}||\phi(d)-u_t(1\otimes d)u_t^*||=0$.
We refer to \cite{TW} for the properties of the strongly self-absorbing C*-algebras.
The following map
$$\Delta_X : (C(X)\otimes\mathcal{O}_\infty)^{\otimes 2}\ni f_1(x)\otimes f_2(y)\mapsto \phi^{-1}(f_1(x)\otimes f_2(x))\in C(X)\otimes \mathcal{O}_\infty$$
gives $K_0(C(X)\otimes \mathcal{O}_\infty)$ a ring structure which coincides with the ring structure of $K^0(X)$ coming from the tensor products of vector bundles.
Let $K^0(X)^\times :=\pm 1+ \tilde{K}^0(X)$ denote the group of the invertible elements.
\begin{thm}[{\cite[{Th. 2.22, 3.8, Lem. 2.8, Cor. 3.9}]{DP}}]\label{DP}
Let $X$ be a connected compact metrizable space, and let $\operatorname{Aut}_0(\mathbb{K}\otimes\mathcal{O}_\infty)$ be the path component of $\operatorname{Aut}(\mathbb{K}\otimes\mathcal{O}_\infty)$ containing $\operatorname{id}_{\mathbb{K}\otimes \mathcal{O}_\infty}$.

1) For two continuous maps $\alpha, \beta : X\to \operatorname{Aut}(\mathbb{K}\otimes\mathcal{O}_\infty)$ which are identified with the $C(X)$-linear isomorphisms of $C(X)\otimes\mathbb{K}\otimes\mathcal{O}_\infty$,
one has
$$K_0(\Delta_X)\circ K_0(\alpha\otimes\beta)([(1_{C(X)}\otimes e\otimes 1_{\mathcal{O}_\infty})^{\otimes 2}]_0)=K_0(\alpha\circ\beta)([1_{C(X)}\otimes e\otimes 1_{\mathcal{O}_\infty}]_0)$$
and the following injective map, whose range is $K^0(X)^\times$, is multiplicative $\colon$
$$[X, \operatorname{Aut}(\mathbb{K}\otimes\mathcal{O}_\infty)]\ni [\alpha]\mapsto [\alpha(1_{C(X)}\otimes e\otimes 1_{\mathcal{O}_\infty})]_0\in K_0(C(X)\otimes \mathbb{K}\otimes\mathcal{O}_\infty).$$

2) The subgroup $[X, \operatorname{Aut}_0(\mathbb{K}\otimes\mathcal{O}_\infty)]\subset [X, \operatorname{Aut}(\mathbb{K}\otimes\mathcal{O}_\infty)]$ is identified with $1+\tilde{K}^0(X)$.

3) The homotopy set $E^1_{\mathcal{O}_\infty}(X):=[X, \operatorname{BAut}(\mathcal{K}\otimes\mathcal{O}_\infty)]$ has a group structure defined by the tensor product $\otimes_{C(X)}$ of locally trivial continuous $C(X)$-algebras of $\mathbb{K}\otimes\mathcal{O}_\infty$,
and $\bar{E}_{\mathcal{O}_\infty}^1(X):=[X, \operatorname{BAut}_0(\mathbb{K}\otimes\mathcal{O}_\infty)]$ is a subgroup of $E^1_{\mathcal{O}_\infty}(X)$.
\end{thm}
For a continuous field $[\mathcal{A}]\in \bar{E}^1_{\mathcal{O}_\infty}(X)$,
one has another field denoted by $\mathcal{A}^{-1}$ satisfying $-[\mathcal{A}]=[\mathcal{A}^{-1}]\in \bar{E}^1_{\mathcal{O}_\infty}(X)$ (i.e., a $C(X)$-linear isomorphism $\mathcal{A}\otimes_{C(X)}\mathcal{A}^{-1}\to C(X)\otimes\mathbb{K}\otimes\mathcal{O}_\infty$ exists).
For a locally trivial continuous field $\mathcal{B}$ of $\mathbb{M}_n(\mathcal{O}_\infty)$,
one has a locally trivial continuous field $\mathbb{K}\otimes\mathcal{B}$ of $\mathbb{K}\otimes\mathcal{O}_\infty$.
Since $\operatorname{Aut}(\mathbb{M}_n(\mathcal{O}_\infty))$ is path connected,
the group homomorphism
$$\eta : \operatorname{Aut}(\mathbb{M}_n(\mathcal{O}_\infty))\ni\sigma\mapsto {\rm id}_\mathbb{K}\otimes\sigma\in \operatorname{Aut}_0(\mathbb{K}\otimes\mathbb{M}_n(\mathcal{O}_\infty))$$
gives a natural map $\operatorname{B}(\eta)_* : [X, \operatorname{BAut}(\mathbb{M}_n(\mathcal{O}_\infty))]\ni [\mathcal{B}]\mapsto [\mathbb{K}\otimes\mathcal{B}]\in \bar{E}^1_{\mathcal{O}_\infty}(X)$.

Let $\operatorname{Aut}_X(\mathcal{A})$ be the group of all $C(X)$-linear isomorphisms of a locally trivial continuous $C(X)$-algebra $\mathcal{A}$ of $\mathbb{K}\otimes\mathcal{O}_\infty$.
We fix an isomorphism $\mathcal{A}\otimes(\mathbb{K}\otimes\mathcal{O}_\infty)\ni f\otimes d\mapsto f\otimes (1_{C(X)}\otimes d)\in \mathcal{A}\otimes_{C(X)}(C(X)\otimes \mathbb{K}\otimes\mathcal{O}_\infty)$.
This gives $K_0(\mathcal{A})$ a $K_0(C(X)\otimes\mathbb{K}\otimes\mathcal{O}_\infty)$-module structure by
$$\cdot [q]_0 : K_0(\mathcal{A})\ni [p]_0\mapsto [p\otimes q]_0\in K_0(\mathcal{A}\otimes_{C(X)}(C(X)\otimes\mathbb{K}\otimes\mathcal{O}_\infty)), \quad [q]_0\in K_0(C(X)\otimes\mathbb{K}\otimes\mathcal{O}_\infty).$$
\begin{lem}\label{g}
In the above setting,
the followings hold $\colon$

1) For every $\alpha\in \operatorname{Aut}_X(\mathcal{A})$,
there is an element $a \in K^0(X)^\times$ with $K_0(\alpha)=\cdot a$,

2) For every $a\in K^0(X)^\times$,
there is an element $\alpha\in \operatorname{Aut}_X(\mathcal{A})$ with $\cdot a=K_0(\alpha)$.
\end{lem}
\begin{proof}
We prove only 1).
Fix an isomorphism $\theta : (\mathcal{A}^{-1}\otimes_{C(X)}\mathcal{A})^{\otimes_{C(X)}2}\to C(X)\otimes\mathbb{K}\otimes\mathcal{O}_\infty$.
One has the following commutative diagram
$$\xymatrix{
\mathcal{A}\otimes(\mathbb{K}\otimes\mathcal{O}_\infty)\ar[d]\ar[r]^{\alpha\otimes{\rm id}}&\mathcal{A}\otimes(\mathbb{K}\otimes\mathcal{O}_\infty)\ar[d]\\
\mathcal{A}\otimes_{C(X)}(C(X)\otimes\mathbb{K}\otimes\mathcal{O}_\infty)\ar[r]^{\alpha\otimes{\rm id}}&\mathcal{A}\otimes_{C(X)}(C(X)\otimes\mathbb{K}\otimes\mathcal{O}_\infty)\\
\mathcal{A}\otimes_{C(X)}(\mathcal{A}^{-1}\otimes\mathcal{A})^{\otimes 2}\ar[u]^{{\rm id}\otimes\theta}\ar[r]^{\quad\;\alpha\otimes({\rm id}_{\mathcal{A}^{-1}}\otimes{\rm id}_{\mathcal{A}})^{\otimes 2}}&\mathcal{A}\otimes_{C(X)}(\mathcal{A}^{-1}\otimes\mathcal{A})^{\otimes 2}.\ar[u]^{{\rm id}\otimes\theta}
}$$
Since the flip automorphism $\sigma : (\mathbb{K}\otimes\mathcal{O}_\infty)^{\otimes 2}\ni x\otimes y\mapsto y\otimes x\in (\mathbb{K}\otimes\mathcal{O}_\infty)^{\otimes 2}$ fixing the minimal projection $(e\otimes 1_{\mathcal{O}_\infty})^{\otimes 2}$ is contained in $\operatorname{Aut}_0((\mathbb{K}\otimes\mathcal{O}_\infty)^{\otimes 2})$ by Theorem \ref{DP}, 2),
one has $K_0(\alpha\otimes({\rm id}_{\mathcal{A}^{-1}}\otimes{\rm id}_{\mathcal{A}})^{\otimes 2})=K_0({\rm id}_{\mathcal{A}}\otimes ({\rm id}_{\mathcal{A}^{-1}}\otimes \alpha\otimes {\rm id}_{\mathcal{A}^{-1}}\otimes{\rm id}_{\mathcal{A}}))$.
We have an element $$a:=[\theta\circ({\rm id}_{\mathcal{A}^{-1}}\otimes \alpha\otimes {\rm id}_{\mathcal{A}^{-1}}\otimes{\rm id}_{\mathcal{A}})\circ\theta^{-1}(1_{C(X)}\otimes e\otimes 1_{\mathcal{O}_\infty})]_0\in K^0(X)^\times $$
satisfying $K_0(\alpha)=\cdot a$.

Using Theorem \ref{DP},
similar argument proves the statement 2).
\end{proof}
\begin{prop}\label{cls}
Let $(X, x_0)$ be a pointed, path connected, finite CW-complex.
For $[\mathcal{A}]\in \operatorname{Im}(\operatorname{B}(\eta)_*)\subset\bar{E}^1_{\mathcal{O}_\infty}(X)$,
we fix an isomorphism $\varphi : \mathcal{A}_{x_0}\cong \mathbb{K}\otimes\mathcal{O}_\infty$.
Then,
the following map is a well-defined bijection
$$\operatorname{B}(\eta)_*^{-1}([\mathcal{A}])\ni [\mathcal{B}]\mapsto [[\rho_\mathcal{B}(e\otimes1_\mathcal{B})]_0]\in \{[p]_0\in K_0(\mathcal{A}) | K_0(\varphi\circ ev_{x_0})([p]_0)=n\}/\sim,$$
where $\rho_\mathcal{B} : \mathbb{K}\otimes\mathcal{B}\to \mathcal{A}$ is an isomorphism satisfying $[(\varphi\circ ev_{x_0}\circ\rho_\mathcal{B})(e\otimes 1_{\mathcal{B}})]_0=n\in K_0(\mathbb{K}\otimes\mathcal{O}_\infty)$.
Here,
the equivalence relation is defined by $[p]_0\sim [r]_0\Leftrightarrow [p]_0\cdot a=[r]_0$ for some $a\in 1+\tilde{K}^0(X)$.
\end{prop}
\begin{proof}
First,
we show taht the map is well-defined.
For two continuous fields $\mathcal{B}_1, \mathcal{B}_2$ with $[\mathcal{B}_1]=[\mathcal{B}_2]\in \operatorname{B}(\eta)_*^{-1}([\mathcal{A}])$ and two $C(X)$-linear isomorphisms $\rho_i : \mathbb{K}\otimes\mathcal{B}_i\to \mathcal{A}$ satisfying $[(\varphi\circ ev_{x_0}\circ\rho_{i})(e\otimes 1_{\mathcal{B}_i})]_0=n$,
we show $[\rho_1(e\otimes 1_{\mathcal{B}_1})]_0\sim [\rho_2(e\otimes 1_{\mathcal{B}_2})]_0$.
Since we have an isomorphism $\gamma : \mathcal{B}_1\to \mathcal{B}_2$,
the following map 
$$\beta :=\rho_2\circ({\rm id}_\mathbb{K}\otimes\gamma)\circ\rho_1^{-1}\in\operatorname{Aut}_X(\mathcal{A})$$
satisfies $\beta\circ\rho_1(e\otimes 1_{\mathcal{B}_1})=\rho_2(e\otimes 1_{\mathcal{B}_2})$,
and Lemma \ref{g} shows $[\rho_1(e\otimes 1_{\mathcal{B}_1})]_0\sim [\rho_2(e\otimes 1_{\mathcal{B}_2})]_0$.
By \cite[{Th. 1.1, 2.7}]{D3},
the element $[p(\mathcal{A})p]$ is sent to $[[p]_0]$ and this map is surjective.

Finally,
we prove the injectivity.
Suppose $[\rho_1(e\otimes 1_{\mathcal{B}_1})]_0\sim [\rho_2(e\otimes 1_{\mathcal{B}_2})]_0$.
Then,
Lemma \ref{g} shows there is a map $\alpha\in\operatorname{Aut}_X(\mathcal{A})$ with $[\alpha(\rho_1(e\otimes 1_{\mathcal{B}_1}))]_0=[\rho_2(e\otimes 1_{\mathcal{B}_2})]_0$.
Now the cancellation of the properly infinite full projections shows $\mathcal{B}_1\cong \mathcal{B}_2$.
\end{proof}
\begin{rem}\label{dd}
For $[\mathcal{A}]=0$,
Proposition \ref{cls} allows us to count the number of the isomorphism classes of the continuous fields of $\mathbb{M}_n(\mathcal{O}_\infty)$ stably isomorphic to $C(X)\otimes\mathbb{K}\otimes\mathcal{O}_\infty$,
which is equal to $|(n+\tilde{K}^0(X))/(1+\tilde{K}^0(X))|$.
\end{rem}


\subsection{The invariant $\mathfrak{b}_{\mathcal{O}_\infty}$}
An automorphism of $E_{n+1}\otimes\mathcal{O}_\infty$ induces an automorphism of $\mathcal{O}_{n+1}\cong (E_{n+1}\otimes\mathcal{O}_\infty)/(\mathbb{K}\otimes\mathcal{O}_\infty)$,
and one has a group homomorphism $q : \operatorname{Aut}(E_{n+1}\otimes\mathcal{O}_\infty)\to \operatorname{Aut}(\mathcal{O}_{n+1})$.
\begin{thm}[{\cite[{Cor. 3.15, Def. 4.1}]{S}}]\label{b}
Let $X$ be a finite CW-complex.

1) The group homomorphism $q : \operatorname{Aut}(E_{n+1}\otimes\mathcal{O}_\infty)\to\operatorname{Aut}(\mathcal{O}_{n+1})$ is a weak homotopy equivalence.

2) For every continuous field $\mathcal{O}$ of $\mathcal{O}_{n+1}$,
one has an exact sequence of $C(X)$-algebras
$$0\to \mathcal{A}\to\mathcal{E}\to\mathcal{O}\to 0,$$
where we denote by $\mathcal{E}$ (resp. $\mathcal{A}$) a continuous field of $E_{n+1}\otimes\mathcal{O}_\infty$ (resp. $\mathbb{K}\otimes\mathcal{O}_\infty$),
and the following map is well-defined $\colon$
$$\mathfrak{b}_{\mathcal{O}_\infty} : [X, \operatorname{BAut}(\mathcal{O}_{n+1})]\ni [\mathcal{O}]\mapsto [\mathcal{A}]\in \bar{E}^1_{\mathcal{O}_\infty}(X).$$
\end{thm}
One has a bijection $(\operatorname{B}(q)_*)^{-1} : [X, \operatorname{BAut}(\mathcal{O}_{n+1})]\ni [\mathcal{O}]\mapsto [\mathcal{E}]\in [X, \operatorname{BAut}(E_{n+1}\otimes\mathcal{O}_\infty)]$ by Theorem \ref{b}, 1),
and the group homomorphism $\operatorname{Aut}(E_{n+1}\otimes\mathcal{O}_\infty)\ni\alpha\mapsto \alpha |_{\mathbb{K}\otimes\mathcal{O}_\infty}\in \operatorname{Aut}_0(\mathbb{K}\otimes\mathcal{O}_\infty)$ induces the map $[X, \operatorname{BAut}(E_{n+1}\otimes\mathcal{O}_\infty)]\ni [\mathcal{E}]\mapsto [\mathcal{A}]\in \bar{E}^1_{\mathcal{O}_\infty}(X)$ (see \cite[{Lem. 3.7}]{S}).

Since $E_{n+1}\otimes\mathcal{O}_\infty$ is KK-equivalent to $\mathbb{C}$,
the map ${\rm id}_{\mathcal{A}^{-1}}\otimes 1 : \mathcal{A}^{-1}\ni f\mapsto f\otimes 1_\mathcal{E}\in \mathcal{A}^{-1}\otimes_{C(X)}\mathcal{E}$ gives the isomorphism $K_0({\rm id}_{\mathcal{A}^{-1}}\otimes 1)$ of their  $K_0$-groups by \cite[{Th. 1.1}]{D3}.

%
\begin{lem}\label{key}
Let $(X, x_0)$ be a pointed, path connected, finite CW-complex.
Let $\mathcal{A}$ and $\mathcal{E}$ be as in Theorem \ref{b}, 2),
and let $\iota : \mathcal{A}\to \mathcal{E}$ be the inclusion map.
We fix isomorphisms $\varphi : \mathcal{E}_{x_0}\cong E_{n+1}\otimes\mathcal{O}_\infty$ and $\pi : (\mathcal{A}^{-1})_{x_0}\cong \mathbb{K}\otimes\mathcal{O}_\infty$.
Then,
there exists a properly infinite full projection $p\in\mathcal{A}^{-1}$ satisfying $K_0(\pi\circ ev_x)([p]_0)=n\in K_0(\mathbb{K}\otimes\mathcal{O}_\infty)$ and
$$[1_{p(\mathcal{A}^{-1})p\otimes_{C(X)}\mathcal{E}}]_0\in \operatorname{Im}(K_0(p(\mathcal{A}^{-1})p\otimes_{C(X)}\mathcal{A})\xrightarrow{K_0({\rm id}\otimes\iota)} K_0(p(\mathcal{A}^{-1})p\otimes_{C(X)}\mathcal{E})).$$
\end{lem}
\begin{proof}
Fix a $C(X)$-linear isomorphism $\psi : C(X)\otimes\mathbb{K}\otimes\mathcal{O}_\infty\to \mathcal{A}^{-1}\otimes_{C(X)}\mathcal{A}$ with $$K_0((\pi\otimes\varphi |_{\mathcal{A}_{x_0}})\circ ev_x)([\psi(1_{C(X)}\otimes e\otimes 1_{\mathcal{O}_\infty})]_0)=1\in K_0((\mathbb{K}\otimes\mathcal{O}_\infty)^{\otimes 2}).$$
By \cite[{Th. 4.2}]{S} and \cite[{Th. 2.11}]{DP2},
the algebra $\mathcal{A}^{-1}$ is isomorphic to a tensor product of a unital $\mathcal{O}_\infty$-stable algebra and $\mathbb{K}$.
Therefore,
one can find a properly infinite full projection $p\in \mathcal{A}^{-1}$ such that $-[p]_0\in K_0(\mathcal{A}^{-1})$ is sent to $K_0({\rm id}_{\mathcal{A}^{-1}}\otimes\iota)([\psi(1_{C(X)}\otimes e\otimes 1_{\mathcal{O}_\infty})]_0)\in K_0(\mathcal{A}^{-1}\otimes_{C(X)}\mathcal{E})$ by the isomorphism $K_0({\rm id}_{\mathcal{A}^{-1}}\otimes 1)$.
Thanks to \cite[{Th. 1.1}]{D3},
two inclusion maps $p(\mathcal{A}^{-1})p\otimes_{C(X)}\mathcal{A}\hookrightarrow \mathcal{A}^{-1}\otimes_{C(X)} \mathcal{A}$ and $p(\mathcal{A}^{-1})p\otimes_{C(X)}\mathcal{E}\hookrightarrow\mathcal{A}^{-1}\otimes_{C(X)}\mathcal{E}$ give isomorphisms of K-groups,
and the statement is now proved.
\end{proof}
\begin{rem}\label{m}
Since $K_0(\mathcal{E})=K^0(X)$,
the image of the $K^0(X)$-module homomorphism $K_0(\mathcal{A})\xrightarrow{K_0(\iota)}K_0(\mathcal{E})$ is an ideal of $K^0(X)$.
In the case of $\operatorname{Tor}(H^*(X), \mathbb{Z}_n)=0$, the ideal is the complete invariant of the continuous field $\mathcal{E}/\mathcal{A}$ (see \cite[{Sec. 2}]{D1} and \cite[{Sec. 4}]{S}).
The element $-[p]_0\in K_0(\mathcal{A}^{-1})$ corresponds to $KK_X(\iota)\in KK_X(\mathcal{A}, \mathcal{E})\cong K_0(\mathcal{A}^{-1})$,
and one can identify $K_0(\iota)$ with the map $K_0(\mathcal{A})\ni [r]_0\mapsto -[r\otimes p]_0\in K_0(\mathcal{A}\otimes_{C(X)}\mathcal{A}^{-1})\cong K^0(X)$.
\end{rem}
By Lemma \ref{key}, we constructs an element $[p(\mathcal{A}^{-1})p]\in [X, \operatorname{BAut}(\mathbb{M}_n(\mathcal{O}_\infty))]$ from $[\mathcal{O}]\in [X, \operatorname{BAut}(\mathcal{O}_{n+1})]$.
In Section 3,
we verify that this gives a natural transformation between two functors $[\cdot, \operatorname{BAut}(\mathcal{O}_{n+1})]$ and $[\cdot, \operatorname{BAut}(\mathbb{M}_n(\mathcal{O}_\infty))]$ defined on the category of the pointed connected finite CW-complexes.
\subsection{The homotopy sets $[X, \operatorname{Aut}(\mathcal{O}_{n+1})]$ and $[X, \operatorname{Aut}(\mathbb{M}_n(\mathcal{O}_\infty))]$}
We briefly recall \cite[{Th. 5.9}]{D4}.
Let $B$ be a unital Kirchberg algebra with path connected $\operatorname{Aut}(B)$,
and let $C_\nu :=\{ f\in C[0, 1]\otimes B | f(0)=0, f(1)\in \mathbb{C}1_B\}$ be the mapping cone of the unital map $\nu : \mathbb{C}\to B$ with the inclusion map $j : SB:= C_0(0, 1)\otimes B\to C_\nu B$.
Let $\alpha, l : (X, x_0)\to (\operatorname{Aut}(B), {\rm id}_B)$ be the base point preserving continuous maps where $l : x\mapsto {\rm id}_B$ is the constant map.
These two maps define an element $\langle \alpha, l\rangle\in KK(C_\nu B, SC_0(X, x_0)\otimes B)$ (see \cite[{p 123}]{D4}),
and the map $[X, \operatorname{Aut}(B)]\ni [\alpha]\mapsto \langle \alpha, l\rangle\in KK(C_\nu, SC_0(X, x_0)\otimes B)$ is bijective.

In the case of $B=\mathbb{M}_n(\mathcal{O}_\infty)$,
one has $KK(C_\nu B, SC_0(X, x_0)\otimes B)=KK(C_\nu B, SC(X)\otimes B)$,
and the map $j^* : KK(C_\nu B, SC(X)\otimes B)\to KK(SB, SC(X)\otimes B)$ maps $\langle \alpha, l\rangle$ to $KK({\rm id}_{C_0(0, 1)}\otimes\alpha )-KK({\rm id}_{C_0(0, 1)}\otimes l)=S(\eta_*([\alpha])-1)$.
Here, the map $S : K^0(X)=KK(B, C(X)\otimes B)\to KK(SB, SC(X)\otimes B)$ is the suspension isomorphism (see Theorem \ref{DP} for the definition of $\eta_*$).
For another map $\beta : (X, x_0)\to (\operatorname{Aut}(B), {\rm id}_B)$,
one has 
\begin{eqnarray*}
\langle \alpha\circ \beta, l\rangle &=&\langle \alpha\circ \beta, \alpha\rangle +\langle \alpha, l\rangle\\
&=&({\rm id}_{C_0(\mathbb{R})}\otimes\alpha)_*(\langle \beta, l\rangle)+\langle\alpha, l\rangle\\
&=&\langle \beta, l\rangle\otimes (KK({\rm id}_{C_0(0, 1)}\otimes\alpha)-1)+\langle \alpha, l\rangle +\langle \beta, l\rangle\\
&=&\langle \alpha, l\rangle+\langle \beta, l\rangle+\langle \beta, l\rangle\cdot (S^{-1}\circ j^*)(\langle\alpha, l\rangle).
\end{eqnarray*}
The product$\langle \beta, l\rangle\cdot (S^{-1}\circ j^*)(\langle \alpha, l\rangle)$ is defined by 
$$KK(C_\nu B, SC(X)\otimes B)\times KK(B, C(X)\otimes B)\to KK(C_\nu B, SC(X\times X)\otimes B)$$
$$\xrightarrow{{\Delta_X}_*} KK(C_\nu B, SC(X)\otimes B).$$ 
\begin{thm}[{\cite[{Th. 6.3, Th. 5.9}]{D4}}]\label{dp}
We write $\operatorname{Ad} : U(C(X)\otimes\mathbb{M}_n(\mathcal{O}_\infty))\ni v\mapsto {\rm Ad}v \in \operatorname{Map}(X, \operatorname{Aut}(\mathbb{M}_n(\mathcal{O}_\infty)))$.

1) There is a short exact sequence of groups $\colon$
$$0\to K_1(C(X)\otimes\mathbb{M}_n(\mathcal{O}_\infty))\otimes\mathbb{Z}_n\xrightarrow{Ad} [X, \operatorname{Aut}(\mathbb{M}_n(\mathcal{O}_\infty))]\xrightarrow{\eta_*}(1+\operatorname{Tor}(K^0(X), \mathbb{Z}_n))^\times\to 0.$$

2)  The multiplication $a\star b := a+b+b\cdot (S^{-1}\circ j^*)(a)$ makes $(KK(C_\nu\mathbb{M}_n(\mathcal{O}_\infty),  SC(X)\otimes \mathbb{M}_n(\mathcal{O}_\infty)), \star)$ a group with the following isomorphism
$$[X, \operatorname{Aut}(\mathbb{M}_n(\mathcal{O}_\infty))]\ni [\beta] \mapsto \langle \beta, l\rangle\in KK(C_\nu \mathbb{M}_n(\mathcal{O}_\infty), SC(X)\otimes\mathbb{M}_n(\mathcal{O}_\infty)).$$
\end{thm}
M. Izumi obtained similar results for $[X, \operatorname{Aut}(\mathcal{O}_{n+1})]$.
Let $\delta : K_1(C(X)\otimes \mathcal{O}_{n+1})\to K^0(X)$ be the index map coming from the exact sequence 
$C(X)\otimes\mathbb{K}\to C(X)\otimes E_{n+1}\to C(X)\otimes\mathcal{O}_{n+1}.$
The map $u : \operatorname{Aut}(\mathcal{O}_{n+1})\to U(\mathcal{O}_{n+1})$ defined by $u(\alpha):=\sum_{i=1}^{n+1}\alpha(S_i)S^*_i$ is a weak homotopy equivalence (see \cite{D2}),
where the isometries $S_i$ are the canonical generators of $\mathcal{O}_{n+1}$.
\begin{thm}[{\cite[{Th. 3.1}]{IS}}]\label{is}
We define a multiplication of $K_1(C(X)\otimes \mathcal{O}_{n+1})$ by $a\diamond b:=a+b-a\cdot \delta (b)$.

1) The map $u_*  : [X, \operatorname{Aut}(\mathcal{O}_{n+1})]\to (K_1(C(X)\otimes\mathcal{O}_{n+1}), \diamond)$ is a group isomorphism.

2) There is a short exact sequence of groups
$$0\to K^1(X)\otimes\mathbb{Z}_n\to [X, \operatorname{Aut}(\mathcal{O}_{n+1})]\xrightarrow{1-\delta} (1+\operatorname{Tor}(K^0(X), \mathbb{Z}_n))^\times\to 0$$
where 
the map $1-\delta$ is defined by $(1-\delta)([\alpha]):=1-\delta([u(\alpha)]_1)$.
\end{thm}
Let $\Sigma X$ be the unreduced suspension of $X$.
Then,
one has a bijection $[\Sigma X, \operatorname{B}G]\to [X, G]$ for every path connected group $G$ (see \cite[{Cor. 8.3}]{H}),
where an element $[\alpha]\in [X, G]$ is sent to the principal $G$-bundle whose clutching function is $\alpha : X\to G$.
We denote by
$$\Gamma_\alpha (\Sigma X)_A :=\{(F_1, F_2)\in (C([0, 1]\times X)\otimes A)^{\oplus 2} | F_i(0)\in 1_{C(X)}\otimes A, F_1(1)=\alpha (F_2(1))\in C(X)\otimes A\}$$
the $C(\Sigma X)$-algebra whose isomorphism class in $[\Sigma X, \operatorname{BAut}(A)]$ corresponds to the element $[\alpha]\in [X, \operatorname{Aut}(A)]$.
We use another C*-algebra defined by
$$M_\alpha (X)_A := \{ f\in C([0, 1]\times X)\otimes A | f(0)\in 1_{C(X)}\otimes A, f(1)=\alpha (f(0))\in C(X)\otimes A\}$$
with the unital map
$M_\alpha (X)_A\ni f\mapsto (f(t), f(0))\in \Gamma_\alpha (\Sigma X).$

For every $[\alpha]\in K^1(X)\otimes\mathbb{Z}_n\subset [X, \operatorname{Aut}(\mathcal{O}_{n+1})]$,
we have a unitary $U_\alpha \in U(C(X)\otimes E_{n+1})$ with $\pi (U_\alpha)=u(\alpha) \in C(X)\otimes\mathcal{O}_{n+1}$,
where $\pi : C(X)\otimes E_{n+1}\to C(X)\otimes\mathcal{O}_{n+1}$ is the quotient map.
We need the following lemma in Section 3.
\begin{lem}[{\cite{IP}}]\label{iz}
Let $X$ be a connected finite CW-complex, and let $\alpha, U_\alpha$ be as above.
Let $v\in U(C(X)\otimes\mathbb{M}_n(\mathcal{O}_\infty))$ be a unitary satisfying $[v]_1=[U_\alpha]_1\in K_1(C(X)\otimes E_{n+1}\otimes\mathbb{M}_n(\mathcal{O}_\infty))$.
Then,
we have $$[1_{M_{\alpha\otimes {\rm Ad}v}(X)_{(\mathcal{O}_{n+1}\otimes\mathbb{M}_n(\mathcal{O}_\infty))}}]_0=0\in K_0(M_{\alpha\otimes {\rm Ad}v}(X)_{(\mathcal{O}_{n+1}\otimes\mathbb{M}_n(\mathcal{O}_\infty))}).$$
In particular,
one has $[1_{(\Gamma_\alpha(\Sigma X)_{\mathcal{O}_{n+1}}\otimes_{C(\Sigma X)}\Gamma_{{\rm Ad}v}(\Sigma X)_{\mathbb{M}_n(\mathcal{O}_\infty)})}]_0=0$.
\end{lem}
\begin{proof}
Let $W$ be the following unitary
$$W:=\left(
\begin{array}{cc}
O_{n+1}&\mathbb{S}^*\\
\mathbb{S}&0_1
\end{array}\right)\in \mathbb{M}_{n+2}(\mathcal{O}_{n+1})\subset C(X)\otimes\mathcal{O}_{n+1}\otimes \mathbb{M}_{n+2}(\mathcal{O}_\infty),$$
where we write $\mathbb{S} :=(S_1, \cdots , S_{n+1})$.
The unitary $W$ is self-adjoint and there is a continuous path of unitary $\{V_t\}_{t\in [0, 1]}\subset C(X)\otimes \mathcal{O}_{n+1}\otimes\mathbb{M}_{n+2}(\mathcal{O}_\infty)$ satisfying $V_0=W, V_1=\alpha\otimes{\rm Ad}(v\oplus 1_2)(W)$.
The unitary $V$ is an element of $M_{\alpha\otimes {\rm Ad}(v\oplus 1_2)}(X)_{(\mathcal{O}_{n+1}\otimes\mathbb{M}_{n+2}(\mathcal{O}_\infty))}$.
It is easy to check that the following inclusion gives isomorphisms of K-groups
$$M_{\alpha\otimes {\rm Ad}v}(X)_{(\mathcal{O}_{n+1}\otimes\mathbb{M}_{n}(\mathcal{O}_\infty))}\ni f\mapsto f\oplus 0_2\in M_{\alpha\otimes {\rm Ad}(v\oplus 1_2)}(X)_{(\mathcal{O}_{n+1}\otimes\mathbb{M}_{n+2}(\mathcal{O}_\infty))},$$
and one has $$[1_{M_{\alpha\otimes {\rm Ad}v}(X)_{(\mathcal{O}_{n+1}\otimes\mathbb{M}_{n}(\mathcal{O}_\infty))}}]_0=\left[\left(
\begin{array}{cc}
1_{n+1}&0\\
0&0_1
\end{array}
\right)\right]_0-\left[\left(
\begin{array}{cc}
O_{n+1}&0\\
0&1_1
\end{array}
\right)\right]_0.
$$
One has
$$V\left(
\begin{array}{cc}
1_{n+1}&0\\
0&0_1
\end{array}
\right)V^*=VW\left(
\begin{array}{cc}
O_{n+1}&0\\
0&1_1
\end{array}
\right)WV^*,$$
and the direct computation yields
\begin{eqnarray*}
V_1W&=&\left(
\begin{array}{cc}
v\oplus 1_1&0\\
0&1_1
\end{array}
\right)\left(
\begin{array}{cc}
O_{n+1}&\alpha(\mathbb{S})^*\\
\alpha(\mathbb{S})&0_1
\end{array}
\right)\left(
\begin{array}{cc}
v^*\oplus 1_1&0\\
0&1_1
\end{array}
\right)\left(
\begin{array}{cc}
O_{n+1}&\mathbb{S}^*\\
\mathbb{S}&0_1
\end{array}
\right)\\
&=&\left(
\begin{array}{cc}
v\oplus 1_1&0\\
0&u(\alpha)
\end{array}
\right)\left(
\begin{array}{cc}
O_{n+1}&\mathbb{S}^*\\
\mathbb{S}&0_1
\end{array}
\right)\left(
\begin{array}{cc}
v^*\oplus 1_1&0\\
0&u(\alpha^*)
\end{array}
\right)\left(
\begin{array}{cc}
O_{n+1}&\mathbb{S}^*\\
\mathbb{S}&0_1
\end{array}
\right)\\
&=&\left(
\begin{array}{cc}
\mathbb{S}^*(\mathbb{S}(v\oplus 1_1)\mathbb{S}^*u(\alpha)^*)\mathbb{S}&0\\
0&u(\alpha) \mathbb{S}(v^*\oplus 1_1)\mathbb{S}^*
\end{array}
\right).
\end{eqnarray*}
Using the following exact sequence
$$K_0(SC(X)\otimes\mathcal{O}_{n+1}\otimes\mathbb{M}_{n+2}(\mathcal{O}_\infty))\hookrightarrow K_0(M_{\alpha\otimes {\rm Ad}(v\oplus 1_2)}(X)_{(\mathcal{O}_{n+1}\otimes\mathbb{M}_{n+2}(\mathcal{O}_\infty))})$$$$\to K_0(\mathcal{O}_{n+1}\otimes\mathbb{M}_{n+2}(\mathcal{O}_\infty)),$$
one can identify $[1_{M_{\alpha\otimes {\rm Ad}v}(X)_{(\mathcal{O}_{n+1}\otimes\mathbb{M}_{n}(\mathcal{O}_\infty))}}]_0$ with
$$[\mathbb{S}(v\oplus 1_1)\mathbb{S}^*u(\alpha)^*]_1=[v]_1-[\pi(U_\alpha)]_1=0\in K_1(C(X)\otimes\mathcal{O}_{n+1}\otimes \mathbb{M}_n(\mathcal{O}_\infty)).$$
\end{proof}

\section{The main theorem}
Let $[\mathcal{E}]\in [X, \operatorname{BAut}(E_{n+1}\otimes \mathcal{O}_\infty)]$ (resp. $[\mathcal{B}]\in [X, \operatorname{BAut}(\mathbb{M}_n(\mathcal{O}_\infty))]$) denote the isomorphism class of a locally trivial continuous $C(X)$-algebra $\mathcal{E}$ (resp. $\mathcal{B}$) whose fiber is $E_{n+1}\otimes\mathcal{O}_\infty$ (resp. $\mathbb{M}_n(\mathcal{O}_\infty)$).
There is a locally trivial continuous $C(X)$-algebra $\mathcal{A}\subset \mathcal{E}$ whose fiber is $\mathbb{K}\otimes\mathcal{O}_\infty$,
and we have $\mathfrak{b}_{\mathcal{O}_\infty}([\mathcal{E}]):=[\mathcal{A}]\in E_{\mathcal{O}_\infty}^1(X)$ (see Theorem \ref{b}).
\begin{thm}\label{w}
Let $(X, x_0)$ be a pointed, path connected, finite CW-complex.
For every $[\mathcal{E}]\in [X, \operatorname{BAut}(E_{n+1}\otimes\mathcal{O}_\infty)]$,
there is a unique element $[\mathcal{B}]\in [X, \operatorname{BAut}(\mathbb{M}_n(\mathcal{O}_\infty))]$ satisfying
$$[\mathbb{K}\otimes\mathcal{B}]=-[\mathcal{A}]=-\mathfrak{b}_{\mathcal{O}_\infty}([\mathcal{E}])\in \bar{E}^1_{\mathcal{O}_\infty}(X),$$
$$[1_{\mathcal{B}\otimes\mathcal{E}}]_0\in \operatorname{Im}(K_0(\mathcal{B}\otimes_{C(X)}\mathcal{A})\xrightarrow{K_0({\rm id}\otimes\iota)}K_0(\mathcal{B}\otimes_{C(X)}\mathcal{E})).$$
Here, we denote by $\iota : \mathcal{A}\hookrightarrow \mathcal{E}$ the inclusion map.
\end{thm}
Lemma \ref{key} implies $[\mathcal{B}]=[p(\mathcal{A}^{-1})p]$,
and the second condition is equivalent to
$$[1_{\mathcal{B}\otimes\mathcal{O}}]_0=0\in K_0(\mathcal{B}\otimes_{C(X)}\mathcal{O}).$$
\begin{lem}\label{W}
Fix a continuous field $\mathcal{E}$ of $E_{n+1}\otimes\mathcal{O}_\infty$.
For two continuous fields $\mathcal{B}_1, \mathcal{B}_2$ of $\mathbb{M}_n(\mathcal{O}_\infty)$ satisfying two conditions in Theorem \ref{w} with respect to $\mathcal{E}$,
we have $[\mathcal{B}_1]=[\mathcal{B}_2]$.
\end{lem}
\begin{proof}
One can find a $C(X)$-linear isomorphism $\gamma : \mathbb{K}\otimes\mathcal{B}_1\to\mathbb{K}\otimes\mathcal{B}_2$ satisfying $[ev_{x_0}\circ\gamma(e\otimes 1_{\mathcal{B}_1})]_0=[e\otimes 1_{({\mathcal{B}_2})_{x_0}}]_0=n\in K_0((\mathbb{K}\otimes{\mathcal{B}_2})_{x_0})$.
Let $\iota : \mathcal{A}\hookrightarrow\mathcal{E}$ be as in Theorem \ref{b}.
Since the inclusion $\mathbb{K}\otimes\mathcal{O}_\infty\to E_{n+1}\otimes\mathcal{O}_\infty$ gives a map $\mathbb{Z}\xrightarrow{-n}\mathbb{Z}$ of $K_0$-groups, 
the preimage of $[1_{\mathcal{B}_i\otimes\mathcal{E}}]_0$ should be a ``rank $- 1$'' projection in $K_0(\mathcal{B}_i\otimes_{C(X)}\mathcal{A})=K^0(X)$.
Now one has an element $a_i\in - 1+\tilde{K}^0(X)\subset K_0(\mathbb{K}\otimes\mathcal{B}_i\otimes_{C(X)}\mathcal{A})^\times$ which is sent to $[e\otimes 1_{\mathcal{B}_i}\otimes 1_{\mathcal{E}}]_0\in K_0(\mathbb{K}\otimes\mathcal{B}_i\otimes_{C(X)}\mathcal{E})$ by the map $K_0({\rm id}_{\mathbb{K}\otimes\mathcal{B}_i}\otimes\iota)$.
The following commutative diagram and Lemma \ref{g} show that there is a map $\alpha\in \operatorname{Aut}_X(\mathbb{K}\otimes\mathcal{B}_2)$ satisfying $[\gamma(e\otimes 1_{\mathcal{B}_1})]_0=[\alpha(e\otimes 1_{\mathcal{B}_2})]_0$, $\cdot(K_0(\gamma\otimes{\rm id})(a_1)\cdot a_2^{-1})=K_0(\alpha)\;$:
$$\xymatrix{
\mathbb{K}\otimes\mathcal{B}_1\otimes_{C(X)}\mathcal{A}\ar[r]^{{\rm id}\otimes\iota}\ar[d]^{\gamma\otimes {\rm id}}&\mathbb{K}\otimes\mathcal{B}_1\otimes_{C(X)}\mathcal{E}\ar[d]^{\gamma\otimes {\rm id}}&\mathbb{K}\otimes\mathcal{B}_1\ar[l]^{\quad\quad{\rm id}\otimes 1}\ar[d]^{\gamma}\\
\mathbb{K}\otimes\mathcal{B}_2\otimes_{C(X)}\mathcal{A}\ar[r]^{{\rm id}\otimes\iota}&\mathbb{K}\otimes\mathcal{B}_2\otimes_{C(X)}\mathcal{E}&\mathbb{K}\otimes\mathcal{B}_2.\ar[l]^{\quad\quad{\rm id}\otimes 1}
}$$
Therefore,
we have $\mathcal{B}_1\cong\mathcal{B}_2$.
\end{proof}
\begin{proof}[{Proof of Theorem \ref{w}}]
Let $(\mathcal{E}_i, \mathcal{A}_i, \mathcal{B}_i),  i=1, 2$ be continuous fields satisfying the conditions of Theorem \ref{w}.
Assume that there is a $C(X)$-linear isomorphism $\phi : \mathcal{E}_1\to \mathcal{E}_2$ (i.e., $[\mathcal{E}_1]=[\mathcal{E}_2]$).
We show $[\mathcal{B}_1]=[\mathcal{B}_2]$.
Since the following diagram commutes,
$$\xymatrix{
\mathcal{B}_1\otimes_{C(X)}\mathcal{A}_1\ar[r]^{{\rm id}\otimes\iota_1}\ar[d]^{{\rm id}\otimes\phi |_{\mathcal{A}_1}}&\mathcal{B}_1\otimes_{C(X)}\mathcal{E}_1\ar[d]^{{\rm id}\otimes\phi}\\
\mathcal{B}_1\otimes_{C(X)}\mathcal{A}_2\ar[r]^{{\rm id}\otimes \iota_2}&\mathcal{B}_1\otimes_{C(X)}\mathcal{E}_2,
}$$
the pair $(\mathcal{E}_2, \mathcal{A}_2, \mathcal{B}_1)$ also satisfies the conditions.
Now Lemma \ref{W} proves the statement.
\end{proof}
By Theorem \ref{w}, the map $t_X : [X, \operatorname{BAut}(E_{n+1}\otimes\mathcal{O}_\infty)]\ni [\mathcal{E}]\mapsto [\mathcal{B}]\in [X, \operatorname{BAut}(\mathbb{M}_n(\mathcal{O}_\infty))]$ is well-defined.
For a base point preserving continuous map $f : (Y, y_0)\to (X, x_0)$ and a continuous field $\mathcal{E}$ over $X$,
one has the pull-back of the continuous field $f^*\mathcal{E}:=C(Y)\otimes_{C(X)}\mathcal{E}$ with a natural homomorphism $\mathcal{E}\ni d\mapsto 1\otimes d\in C(Y)\otimes_{C(X)}\mathcal{E}=f^*\mathcal{E}$.
Since one has $f^*(\mathcal{B}\otimes_{C(X)}\mathcal{E})\cong f^*\mathcal{B}\otimes_{C(Y)}f^*\mathcal{E}$,
it is easy to check that the map $t_X$ is natural with respect to $X$.
\begin{dfn}
Let $q : \operatorname{Aut}(E_{n+1}\otimes\mathcal{O}_\infty)\to \operatorname{Aut}(\mathcal{O}_{n+1})$ be the group homomorphism which is a weak homotopy equivalence.
Let $\mathcal{C}_0$ be the category whose objects are pointed, path connected, finite CW-complexes,
and morphisms are the base point preserving continuous maps.
Let $\mathcal{S}$ be the category of sets with a distinguished element and maps preserving the distinguished elements.
For $(X, x_0)\in \mathcal{C}_0$,
we regard $(X, x_0)\mapsto [X, \operatorname{BAut}(\mathcal{O}_{n+1})]$ and $(X, x_0)\mapsto [X, \operatorname{BAut}(\mathbb{M}_n(\mathcal{O}_\infty))]$ as contravariant functors from $\mathcal{C}_0$ to $\mathcal{S}$ where the distinguished element is the homotopy class of the constant map.
We define a natural transformation by
$$T_X :=t_X\circ (\operatorname{B}(q)_*)^{-1} : [X, \operatorname{BAut}(\mathcal{O}_{n+1})]\to [X, \operatorname{BAut}(\mathbb{M}_n(\mathcal{O}_\infty))].$$
\end{dfn}
\begin{prop}\label{B}
One has $T_X([\mathcal{O}])=[\mathcal{B}]$ if and only if $[1_{\mathcal{B}\otimes\mathcal{O}}]_0=0\in K_0(\mathcal{B}\otimes_{C(X)}\mathcal{O})$.
\end{prop}
\begin{proof}
Assume $[1_{\mathcal{B}\otimes\mathcal{O}}]_0=0$.
Since there is an exact sequence $\mathcal{A}\to\mathcal{E}\to\mathcal{O}$ as in Theorem \ref{b}, we have an element $a\in K_0(\mathcal{B}\otimes_{C(X)}\mathcal{A})$ which is sent to $[1_{\mathcal{B}\otimes\mathcal{E}}]_0\in K_0(\mathcal{B}\otimes_{C(X)}\mathcal{E})$.
Therefore,
the element $a$ is ``rank $-1$'' projection,
and \cite[{Th. 4.2}]{DP} implies $-[\mathcal{A}]=[\mathbb{K}\otimes\mathcal{B}]\in \bar{E}_{\mathcal{O}_\infty}^1(X)$ (i.e., $T_X([\mathcal{O}])=[\mathcal{B}]$).
\end{proof}
\begin{rem}
One has $[X, \operatorname{BAut}(\mathcal{O}_{n+1})]_0=[X, \operatorname{BAut}(\mathcal{O}_{n+1})]$ and $[X, \operatorname{BAut}(\mathbb{M}_n(\mathcal{O}_\infty))]_0=[X, \operatorname{BAut}(\mathbb{M}_n(\mathcal{O}_\infty))]$ as mentioned in Section 2.1.
The map $T_X$ is bijective for every $(X, x_0)\in\mathcal{C}_0$ if and only if $T_{S^k}$ is bijective for every $k\geq 1$ by Brown's representability theorem \cite[{Lem. 1.5}]{Br}.
\end{rem}
\begin{lem}\label{grr}
Let $SX$ be the reduced suspension of a connected finite CW-complex $(X, x_0)$.
For a path connected group $G$,
we have a natural group isomorphism $[SX, \operatorname{B}G]_0\to [X, G]$.
\end{lem}
\begin{proof}
Since $G$ is path connected, one has $[SX, \operatorname{B}G]_0=[SX, \operatorname{B}G]$.
For a CW-complex,
the quotient map $\Sigma X\to SX$ is a homotopy equivalence,
and the map $[SX, \operatorname{B}G]\to [\Sigma X, \operatorname{B}G]$ is bijective.
By \cite[{Cor. 8.3}]{H},
we have a bijection $[\Sigma X, \operatorname{B}G]\to [X, G]$,
where the homotopy class of the map $\alpha : X\to G$ is sent to the isomorphism class of the principal $G$-bundle on $\Sigma X$ whose clutching function over $X$ is the map $\alpha$.
It is easy to check that the composition of the above maps is a group homomorphism.
\end{proof}
Since two classifying spaces $\operatorname{BAut}(\mathcal{O}_{n+1})$ and $\operatorname{BAut}(\mathbb{M}_n(\mathcal{O}_\infty))$ have countable homotopy groups,
there are CW-complexes $Y_1, Y_2$ with countably many cells and two weak homotopy equivalences $Y_1\to \operatorname{BAut}(\mathcal{O}_{n+1})$, $Y_2\to\operatorname{BAut}(\mathbb{M}_n(\mathcal{O}_\infty))$ (see \cite[{p 188}]{AT}).
\begin{cor}\label{gi}
The map $T_{SX} : [SX, \operatorname{BAut}(\mathcal{O}_{n+1})]\to [SX, \operatorname{BAut}(\mathbb{M}_n(\mathcal{O}_\infty))]$ gives a group homomorphism $[X, \operatorname{Aut}(\mathcal{O}_{n+1})]\to [X, \operatorname{Aut}(\mathbb{M}_n(\mathcal{O}_\infty))]$,
and, for $X=S^{k-1}$,
the map $T_{S^k} : [S^k, \operatorname{BAut}(\mathcal{O}_{n+1})]\to [S^k, \operatorname{BAut}(\mathbb{M}_n(\mathcal{O}_\infty)]$ is identified with ${\rm id}_{K^1(S^{k-1})\otimes\mathbb{Z}_n}$.
\end{cor}
\begin{proof}
First, we show that the map $T_{SX}$ is a group homomorphism.
By Lemma \ref{grr},
it suffices to construct a map $f : Y_1\to Y_2$ representing the natural transformation
$$[SX, Y_1]_0\to [SX, \operatorname{BAut}(\mathcal{O}_{n+1})]_0\xrightarrow{T_{SX}}[SX, \operatorname{BAut}(\mathbb{M}_n(\mathcal{O}_\infty))]_0\to [SX, Y_2]_0.$$
Now \cite[{Lem. 1.7}]{Br} shows the existence of such a map $f$.

Next,
we show that the map $T_{S^k}$ is bijective.
Fix $v\in U(C(X)\otimes\mathbb{M}_n(\mathcal{O}_\infty))$ and $[\alpha]\in [S^{k-1}, \operatorname{Aut}(\mathcal{O}_{n+1})]$ with $[v]_1\otimes \bar{1}=[U_\alpha]_1\otimes \bar{1}\in K^1(S^{k-1})\otimes\mathbb{Z}_n=[S^{k-1}, \operatorname{Aut}(\mathcal{O}_{n+1})]$ as in Lemma \ref{iz}.
Lemma \ref{iz} and Proposition \ref{B} implies $T_{S^k}([\Gamma_{\alpha}(\Sigma S^{k-1})_{\mathcal{O}_{n+1}}])=[\Gamma_{{\rm Ad}v}(\Sigma S^{k-1})_{\mathbb{M}_n(\mathcal{O}_\infty)}]$ (i.e., $T_{S^k}([\alpha])=[{\rm Ad}v]$).
Therefore, the map $T_{S^k}$ is identified with $${\rm id} : K^1(S^{k-1})\otimes\mathbb{Z}_n\ni [U_\alpha]_1\otimes\bar{1}\mapsto [v]_1\otimes \bar{1}\in K^1(S^{k-1})\otimes\mathbb{Z}_n.$$
\end{proof}
Now we have shown our main theorem.
\begin{thm}\label{M}
For every path connected finite CW-complex $X$,
the map $T_X$ is bijective.
\end{thm}
\begin{cor}\label{si}
We have $-\mathfrak{b}_{\mathcal{O}_\infty}=\operatorname{B}(\eta)_*\circ T_X$ and $-\operatorname{Im}(\mathfrak{b}_{\mathcal{O}_\infty})=\operatorname{Im}(\operatorname{B}(\eta)_*)$.
The map $T_{SX}$ gives a group isomorphism,
and the following diagram commutes $\colon$
$$\xymatrix{
K^1(X)\otimes\mathbb{Z}_n\ar[r]\ar@{=}[d]&[X, \operatorname{Aut}(\mathcal{O}_{n+1})]\ar[r]^{1-\delta\quad\quad}\ar[d]^{T_{SX}}&(1+\operatorname{Tor}(K^0(X), \mathbb{Z}_n))^\times\ar[d]^{\cdot^{-1}}\\
K^1(X)\otimes\mathbb{Z}_n\ar[r]^{\rm Ad \quad\quad}&[X, \operatorname{Aut}(\mathbb{M}_n(\mathcal{O}_\infty))]\ar[r]^{\eta_*\quad\quad}&(1+\operatorname{Tor}(K^0(X), \mathbb{Z}_n))^\times,
}$$
where the right vertical map sends an invertible element of the K-theory ring to its inverse.
\end{cor}
Thanks to Remark \ref{dd},
now we have another proof of the result \cite[{Th. 5.3}]{D1}.
\begin{cor}[{\cite[{Th. 5.3}]{D1}, \cite[{Th. 4.14, Rem. 4.13, 4.15}]{S}}]\label{ds}
The cardinality of $\mathfrak{b}_{\mathcal{O}_\infty}^{-1}(0)$ is equal to $|(n+\tilde{K}^0(X))/(1+\tilde{K}^0(X))|$.
In particular,
if $\operatorname{Tor}(H^{2k+1}(X), \mathbb{Z}_n)=0$ for $k\geq 1$,
one has $|[X, \operatorname{BAut}(\mathcal{O}_{n+1})]|=|(n+\tilde{K}^0(X))/(1+\tilde{K}^0(X))|$.
\end{cor}

\end{document}